\documentclass[hidelinks,onefignum,onetabnum]{siamart220329}
\usepackage{amsfonts}
\usepackage{graphicx}
\usepackage{epstopdf}
\usepackage{algorithmic}
\ifpdf
  \DeclareGraphicsExtensions{.eps,.pdf,.png,.jpg}
\else
  \DeclareGraphicsExtensions{.eps}
\fi
\usepackage{acronym}
\usepackage{pgfplots}
\usepgfplotslibrary{groupplots}
\pgfplotsset{compat=1.17}
\usepackage[font=footnotesize,labelformat=simple]{subcaption} 
\usepackage{nomencl}
\usepackage{etoolbox}

\renewcommand{\nompreamble}{The following list introduces various symbols used within the manuscript}

\newsiamremark{remark}{Remark}
\newsiamthm{lem}{Lemma}
\newsiamthm{prop}{Proposition}
\newsiamremark{assum}{Assumption}

\headers{Variance bounds for biochemical processes}{Giovanni Pugliese Carratelli and Ionnis Lestas}
\title{Variance bounds for a class of {biochemical birth/death like processes} via a discrete expansion and spectral properties of the master equation
\thanks{{{{This manuscript is a significantly extended version of conference paper \cite{PuglieseCarratelli2020} and includes the proofs of the results, {which are centred  upon spectral properties of the master equation used in conjunction with the Newton Series expansion,} and includes also additional discussion.}}}}}

\author{Giovanni Pugliese Carratelli\thanks{Department of Engineering, University of Cambridge, Trumpington Street - CB2 1PZ Cambridge, United Kingdom (\email{\{gp459,icl20\}@cam.ac.uk}).}
\and Ioannis Lestas\footnotemark[2]}

\usepackage{amsopn}

\acrodef{CME}       [CME]           {Chemical Master Equation}
\acrodef{LMI}       [LMI]           {Linear Matrix Inequality}
\acrodef{ODE}       [ODE]           {Ordinary Differential Equation}
\acrodef{SDE}       [SDE]           {Stochastic Differential Equation}
\acrodef{LNA}       [LNA]           {Linear Noise Approximation} 

\ifpdf
\hypersetup{
  pdftitle={Variance bounds for a class of biochemical birth/death like process for
 reactions via a discrete expansion and spectral properties of the master equations},
  pdfauthor={Giovanni Pugliese Carratelli and Ioannis Lestas}
}
\fi

\providetoggle{nomsort}
\settoggle{nomsort}{true} 

\makeatletter
\iftoggle{nomsort}{%
    \let\old@@@nomenclature=\@@@nomenclature
        \newcounter{@nomcount} \setcounter{@nomcount}{0}%
        \renewcommand\the@nomcount{\two@digits{\value{@nomcount}}}
        \def\@@@nomenclature[#1]#2#3{
          \addtocounter{@nomcount}{1}%
        \def\@tempa{#2}\def\@tempb{#3}%
          \protected@write\@nomenclaturefile{}%
          {\string\nomenclatureentry{\the@nomcount\nom@verb\@tempa @[{\nom@verb\@tempa}]%
          \begingroup\nom@verb\@tempb\protect\nomeqref{\theequation}%
          |nompageref}{\thepage}}%
          \endgroup
          \@esphack}%
      }{}
\makeatother

\begin{document}
\maketitle
\begin{abstract}
We consider a class of birth/death like process corresponding to coupled biochemical reactions and consider the problem of quantifying the variance of the molecular species in terms of the rates of the reactions.
In particular, we address this problem in a configuration where a species is formed with a rate that depends nonlinearly on another spontaneously formed species. By making use of an appropriately formulated expansion based on the Newton series, in conjunction with spectral properties of the master equation, we derive an analytical expression that provides a hard bound for the variance. We show that this bound is exact when the propensities are linear, with numerical simulations demonstrating that this bound is also very close to the actual variance. An analytical expression for the covariance of the species is also derived.
\end{abstract}


\section{Introduction}\label{sec.Introduction}
Biochemical reactions are often modeled as birth/death like processes with rates that depend on other molecular species.
This leads to continuous time jump Markov process, that can be analysed by means of their forward Kolmogorov equation, known as the \ac{CME}, or via stochastic simulations \cite{Gillespie2007}.
Quantifying the moments of molecular species is important as this affects the functionality of various biochemical pathways within cells. Assessing variance in enzymatic networks \cite{Heinrich_2002} can for instance help understand how enzyme kinetics pathways are an effective signalling configuration and clarify how cells modulate protein activity in response to external stimuli. When the rates of the reactions depend nonlinearly on other molecular species however,  analytical expressions for the moments are in general not feasible. There are many computational tools that have been developed in the literature for approximating the moments, based on either time simulations or  numerical approximations to solutions of the master equation and its corresponding moments.  The latter include the derivation of expressions that provide characterisations of the stationary distributions \cite{Cappelletti2021}, computational methods for level-dependent quasi-birth-and-death processes with finite \cite{Gaver1984} and infinite number of states \cite{Kharoufeh2011}, and the Finite State Projection algorithm \cite{Munsky2006}. Optimization based formulations by means of semidefinite programming have been addressed in \cite{Kuntz2019}, \cite{Lasserre2009}, \cite{Ghusinga2017,Dowdy2018,Sakurai_2018}, and approximate moment dynamics with closure properties have been studied in \cite{Hespanha2005}, \cite{Singh2011} and \cite{Smadbeck2013}.

Nevertheless, obtaining analytical expressions for the variance in terms of biological observables, such as expectations of the rates and molecule numbers, is important as this can facilitate the physical understanding of such processes and also motivate design in synthetic biology. \acl{LNA} \cite{Kampen2007} is a tool frequently used in this context, whereby the nonlinear reaction rates are linearised about the equilibrium values for the mean, \emph{i.e.} the first two terms of their Taylor series expansion are retained. Such approximations are valid when the reaction rates are approximately nonlinear or when there are large molecule numbers. They are, however, usually not very accurate in intermediate and low molecule numbers as is often the case in biology.

In this paper we present an alternative approach whereby {an appropriately formulated expansion based on the \emph{Newton series} 
\cite{NewtonEng}, \cite{Newton}}, a discrete analog of the Taylor series, is used to derive a hard {lower} bound for the variance for a class of biochemical reactions. In particular, the bound is derived for a species that is formed with bursts with a nonlinear rate that is a function of another spontaneously formed species.

{{A key idea in our analysis is to exploit} the eigenfunctions of the considered master equation {and show} that {these can form} a basis for the {Newton series} {expansion}. {The latter is then used to obtain an analytical expression that provides a bound} for the variance {in the molecule numbers}. We also show that the bound is exact when the reaction rates are linear.}

{Furthermore, a side result associated with the covariance of the {molecule number of the species under consideration is derived, whereby the {\em Newton series} expansion is exploited to provide an exact analytical expression for the covariance.}}

{In the paper we also} use numerical {computations of the variance} to investigate how close the bound {derived} is to the actual variance, and we find that, unlike {with {the} \acl{LNA},}
this is very close to the true variance even in regimes where the reaction rate is highly {nonlinear.}

{The paper is structured as follows. In Section \ref{sec.Preliminaries} we provide the problem formulation. The main results are given in Section \ref{sec.Results}. The proofs of the results are provided in the Section \ref{sec.MYAppendix}.
A numerical evaluation of the results is provided in Section \ref{sec.Examples}. Finally, conclusions are drawn in Section \ref{sec.Conclusions}. }

\begin{thenomenclature} 
\nomgroup{Symbols}
  \item [{$\mathbb{R}_{>}$}]\begingroup Set of positive real numbers $\{x \in \mathbb{R}: x > 0\}$\nomeqref {1.0}\nompageref{2}
  \item [{$\mathbb{Z}_{\geq}$}]\begingroup Set of non-negative integers $\{0, 1, 2, 3 \ldots \}$\nomeqref {1.0}\nompageref{2}
  \item [{$\mathbb{R}_{\geq}$}]\begingroup Set of non-negative real numbers $\{x \in \mathbb{R}: x \geq 0\}$\nomeqref {1.0}\nompageref{2}
  \item [{$\mathbb{Z}_>$}]\begingroup Set of positive integers $\{1, 2, 3 \ldots \}$\nomeqref {1.0}\nompageref{2}
  \item [{$\mathbb{P}_Q(q)$}]\begingroup Probability {integer-valued random variable $Q$ takes value $q$}\nomeqref {1.0}\nompageref{2}
  \item [{$A \sim \mathcal{P}(x)$}]\begingroup Random variable $A$ has Poisson probability distribution with mean $x$\nomeqref {1.0}\nompageref{2}
  \item [{$\mathbb{R}[A]$}]\begingroup Set of polynomials in $A$ with {real coefficients}\nomeqref {1.0}\nompageref{2}
  \item [{$\mathbb{E}[X]$}]\begingroup Expectation of random variable $X$\nomeqref {1.0}\nompageref{2}
  \item [{$\mathbb{E}_{\mathcal{P}}[\cdot]$}]\begingroup {Expectation with} respect to {probability distribution $\mathcal{P}$.}\nomeqref {1.0}\nompageref{2}
  \item [{$var(X)$}]\begingroup Variance of random variable $X$\nomeqref {1.0}\nompageref{2}
  \item [{$cov(X,Y)$}]\begingroup Covariance of random variables $X$ and $Y$\nomeqref {1.0}\nompageref{2}
  \item [{$\rho_{AB}$}]\begingroup {Pearson correlation} coefficient between random variables $A$ and $B$, \emph{i.e.}\\ {$\rho_{AB}= \frac{cov(A,B)}{\sqrt{var(A)}\sqrt{var(B)}}$}\nomeqref {1.0}\nompageref{2}
  \item [{$dom(f)$}]\begingroup Domain of function $f$\nomeqref {1.0}\nompageref{2}
  \item [{$\Delta_k[\cdot]$}]\begingroup $k$ step finite difference of {a function $f:\mathbb{Z}\to\mathbb{Z}$, defined as $\Delta_k[f](A) := f(A+k) - f(A)$}\nomeqref {1.0}\nompageref{2}
  \item [{$\Delta_k^p[\cdot]$}]\begingroup $pk$ step finite difference operator {defined for $p>1, p \in \mathbb{Z}_>,$ via the recursion $\Delta_k^1[\Delta_k^{p-1}[\cdot]]$ with $\Delta_k^0[.]$ the identity operator, and $\Delta_k^1[.]:=\Delta_k[.]$.}\nomeqref {1.0}\nompageref{2}
  \item [{$(A)_k$}]\begingroup Falling factorial of $A \in \mathbb{Z}_{\geq}$ with $k \in \mathbb{Z}_{\geq}$: if $k \neq 0$ then $(A)_k = \prod_{n=0}^{k-1}A-n$, if $k = 0$ then $(A)_0 = 1$\nomeqref {1.0}\nompageref{2}
  \item [{$\mathbb{I}$}]\begingroup Identity  operator: $\mathbb{I}[X] = X$\nomeqref {1.0}\nompageref{2}
  \item [{$\cdot^{T}$}]\begingroup Matrix transpose operator\nomeqref {1.0}\nompageref{2}
  \item [{$\delta_{a,b}$}]\begingroup $\delta_{a,b} = 1$, if $a=b$, $0$ otherwise\nomeqref {1.0}\nompageref{2}
  \item [{$\mathcal{U}\{a,b\}$}]\begingroup Discrete uniform probability distribution with {support $\{a, \hdots, b\}$} with $a,b \in \mathbb{Z}_{\geq}$\nomeqref {1.0}\nompageref{2}
  \item [{$\mathcal{G}(p)$}]\begingroup Geometric probability distribution, with parameter {$p\in(0,1]$}\nomeqref {1.0}\nompageref{2}
  \item [{{$\mathcal{B}(n,p)$}}]\begingroup {Binomial probability distribution, {with parameters} {$p\in[0,1]$} and {$n \in \mathbb{Z}_{>} $}}\nomeqref {1.0}\nompageref{2}

\end{thenomenclature}

\section{Problem formulation}\label{sec.Preliminaries}
We {consider} the following system {of coupled biochemical reactions}
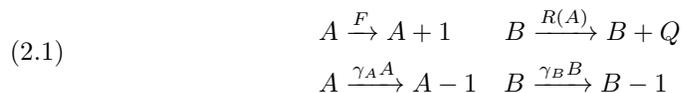
\begin{equation}\label{eq.ChemicalEquationProcess}
\begin{split}
  A & \xrightarrow[]{F} A+1 \qquad B \xrightarrow[]{R(A)} B+Q \\
  A & \xrightarrow[]{\gamma_A A} A-1 \quad B \xrightarrow[]{\gamma_B B} B-1
\end{split}
\end{equation}
\noindent where $A$ and $B$ are {two biochemical} species, {with species $A$ affecting the rate of formation of species $B$ via a nonlinear function $R(A)$. This will be modelled as a continuous time jump Markov Process, where its forward Kolmogorov equation, known as the \acf{CME}, is provided in \eqref{eq.ChemicalMasterEquationProcess} below.}

More precisely, random variable $A(t)$ denotes the number of molecules of species $A$ at time $t$ and similarly random variable $B(t)$ denotes the number of molecules of species $B$ at time $t$. The random variable $Q$ denotes the increase in {the number of molecules of species $B$} when a birth takes place.
The parameter $F \in \mathbb{R}_{>}$ is a constant that denotes the rate at which the species $A$ is produced, $R: \mathbb{Z}_{\geq} \rightarrow \mathbb{R}_{\geq}$ {is a function that determines} the rate of production of the species $B$ and constants $\gamma_A \in \mathbb{R}_{>}$ and  $\gamma_B \in \mathbb{R}_{>}$ represent the death rate of each molecule of species $A$ and $B$ respectively.

For any $a,b \in \mathbb{Z}_{\geq}$, we denote by $\mathbb{P}(a,b,t)$ the probability $A(t)=a$, $B(t)=b$. {Random variable $Q$ takes values in $\mathbb{Z}_>$ and we denote by $\mathbb{P}_{Q}(q)$ the probability $Q=q$. In our results random variable $Q$ can have any discrete probability distribution defined on the support $\mathbb{Z}_>$ satisfying {Assumption \ref{assum.finiteMQB} stated below.}}
The \ac{CME}, a version of the Chapman Kolmogorov equation for Markov processes, for system \eqref{eq.ChemicalEquationProcess} is
\begin{equation}
\begin{split}\label{eq.ChemicalMasterEquationProcess}
	\dfrac{\partial \mathbb{P}(a,b,t)}{\partial t}&= \gamma_A [(a+1)\mathbb{P}(a+1,b,t) - a\mathbb{P}(a,b,t)] + F[\mathbb{P}(a-1,b,t) -\mathbb{P}(a,b,t)]\\
  & + \gamma_B[(b+1)\mathbb{P}(a,b+1,t)-b\mathbb{P}(a,b,t)]\\
  &+ R(a)\left[\sum_{q=1}^{b}\mathbb{P}_{Q}(q)\mathbb{P}(a,b-q,t) - \mathbb{P}(a,b,t)\right]
\end{split}
\end{equation}

Mathematical models as in \eqref{eq.ChemicalEquationProcess} are relevant in order to quantify noise levels where a spontaneously formed species $A$ affects the rate of formation of species $B$.
\noindent {{Such models are} of particular interest for studies of various {biological} processes within cells. 
{These appear, for example,} in cellular {signalling} \cite{Heinrich_2002} where external receptors stimulate enzymatic reactions that modulate protein activity in response to external stimuli. Quantifying variance and correlation in this {signaling cascade} {contributes significantly in improving our understanding of} the efficiency of {these signaling} pathways and {the extent of noise amplification in these processes} \cite{Klipp2006}.
The variable $Q$ also captures an effect known as \emph{bursts} that has been observed experimentally in gene transcription \cite{golding2005real} and in the expression of proteins \cite{Paulsson2005}. Such a model has been shown in the literature to follow when {there is a timescale {separation} between various processes, such as when the} degradation rate of the mRNAs is much larger than that of proteins, which is often the case in practice \cite{shahrezaei2008analytical}.}

The problem we address is to find an expression of the mean and variance of $B$ at steady state, expressed in terms of the expected value of the rates of the reactions. This is a non-trivial problem due to the nonlinear rate $R(A)$.

In this manuscript we derive an analytical expression that provides a bound on the variance of $B$, by making use of a discrete expansion of $R(A)$ based on the \emph{Newton} series \cite{MilneThomson1981,Jordan1965} and the eigenfunctions of the \ac{CME}. The bound becomes exact when $R(A)$ is linear. Furthermore, we show via numerical {computations of the variance that the bound} is very close to the true variance even in regimes where $R(A)$ is highly nonlinear, in contrast to the \ac{LNA} for the variance.

It should be noted that the methodology used to derive the bound, differently from the \ac{LNA}, exploits the discrete states of the {system} via the {Newton series}.

\noindent The Newton series has also been recently studied in the field of quantum mechanics \cite{Koenig2021} while filtering problems associated with \eqref{eq.ChemicalEquationProcess} have also been studied in the literature. The optimal causal filter for estimating $A$ when $B$ is observed was derived in \cite{Snyder1975}, the average squared estimation error was quantified in \cite{Hinczewski2014}, and a close to optimal version was deployed \emph{in vitro} in \cite{Zechner2016}.

The results that will be presented are associated with the equilibrium distribution of species $A$ and $B$ and we therefore make the following assumption \cite{Gupta2014}.
\begin{assum}\label{assum.Stationaity} System
	\eqref{eq.ChemicalMasterEquationProcess} reaches a {stationary
	distribution} denoted by $\mathbb{P}(a,b)$ \begin{equation} \lim_{t
	\rightarrow+\infty} \mathbb{P}(a,b,t) = \mathbb{P}(a,b) \end{equation}
\end{assum}

	Note that from the \ac{CME} \eqref{eq.ChemicalMasterEquationProcess}
	$\mathbb{P}(a,b)$ satisfies the following equation
	\begin{multline}\label{eq.ChemicalMasterEquationProcessStationary}
\gamma_A [(a+1)\mathbb{P}(a+1,b) -a\mathbb{P}(a,b)] + F[\mathbb{P}(a-1,b)
-\mathbb{P}(a,b)] \\ + \gamma_B[(b+1)\mathbb{P}(a,b+1)-b\mathbb{P}(a,b)]+
R(a)\left[\sum_{q=1}^{b}\mathbb{P}_{Q}(q)\mathbb{P}(a,b-q) -
\mathbb{P}(a,b)\right] = 0
\end{multline}

For convenience in the notation we denote by $A$ the random variable representing the number of molecules of species $A$ at equilibrium, and its probability distribution is denoted by
$\mathbb{P}_A$. Similarly, we denote by $B$ the random variable representing the number of molecules of species $B$ at equilibrium.
It should also be noted that
\begin{equation}\label{def:A_Poisson}
	A \sim \mathcal{P}(F/\gamma_A)
\end{equation}
\emph{i.e.} random variable $A$ has a Poisson distribution with expected value
$F/\gamma_A$ which is a known result that follows from the fact that $F$ and $\gamma_A$ are constant, see \cite{Kampen2007}.
{
	\begin{remark}\label{rmk.finitesupport}
		We would like to note that when random variable $Q$ {has a finite support} and $\mathbb{E}[R(A)]$ is finite, the validity of Assumption \ref{assum.Stationaity} follows from the ergodicity results in \cite{Cappelletti2021}. In particular,  system \eqref{eq.ChemicalMasterEquationProcess} belongs in this case in the class of reactions for which the  \cite[Theorem 6.1]{Cappelletti2021} holds, and Assumption \ref{assum.Stationaity} follows from this ergodicity result.
\end{remark}}

{{We also assume that random variables $B$ and $Q$ have finite first and second moments as stated below.}}
\begin{assum}\label{assum.finiteMQB}
	{The first and second moments of random variables $B$ and $Q$ are finite, {i.e. $\mathbb{E}[B^2]<\infty $, $\mathbb{E}[Q^2]<\infty$ which implies also $\mathbb{E}[B]< \infty$, $\mathbb{E}[Q]< \infty$. We also assume $\lim_{a\to\infty}a\sum_{b=0}^\infty b^2\mathbb{P}(a,b)<\infty$ which is implied by $\mathbb{E}[AB^2]<\infty$.}}
\end{assum}
\section{Results}\label{sec.Results}
{Before we present our results, we {define the following quantities:}}
\begin{definition}\label{def.s0s1}
{Let $A$ be a random variable with Poisson distribution as specified in \eqref{def:A_Poisson}. The quantities $\sigma_0$ and $\sigma_1$ are defined as {follows}}
\begin{equation}\label{eq.Sigma0Sigma1Statement}
\sigma_0 = \mathbb{E}_{\mathcal{P}}[R(A)], \quad \sigma_1 = - {\sigma_0} + \frac{\mathbb{E}_{\mathcal{P}}[AR(A)]}{{\mathbb{E}_{\mathcal{P}}[A]}}
\end{equation}
\end{definition}
{

{The quantities defined above will be needed in the variance bound {presented} in Proposition \ref{prop.VarianceBoundNonLinearProcess} and the analytical expression for the covariance in Proposition \ref{eq.COVarianceBoundNonLinearProcess}. Their significance is also discussed in Remark \ref{rmk.Varboundcoeff} after Proposition \ref{prop.VarianceBoundNonLinearProcess} is stated.}

It should be noted that the parameters $\sigma_0$, $\sigma_1$ are finite under the assumption that the variance of $B$ is bounded, as follows from the variance bound \eqref{eq.VarianceDownstreamProcessInequality} in Proposition \ref{prop.VarianceBoundNonLinearProcess}. Also in Definition \ref{def.s0s1} $\mathbb{E}_{\mathcal{P}}[\cdot]$ is the expectation with respect to a Poisson distribution with a known mean and can be computed to an {arbitrarily} high precision.

We would also like to note that no assumptions are imposed on function $R(A)$ apart from the fact that this takes values in {$\mathbb{R}_{\geq}$} and leads to a random variable $B$ in \eqref{assum.Stationaity}, \eqref{eq.ChemicalMasterEquationProcessStationary} with finite first two moments.
}

We now give our main result which is an expression that provides a lower bound on the variance of $B$.
\begin{prop}\label{prop.VarianceBoundNonLinearProcess}
Consider the system in \eqref{eq.ChemicalMasterEquationProcessStationary} {
{with Assumption \ref{assum.finiteMQB} satisfied}}. 
{Then the} following inequality holds
  \begin{equation}\label{eq.VarianceDownstreamProcessInequality}
    \begin{split}
    & var(B)>\frac{\sigma_0(\gamma_A+\gamma_B)(\mathbb{E}[Q^2] + \mathbb{E}[Q]) + \sigma_1^22(\mathbb{E}[Q])^2 \mathbb{E}_{\mathcal{P}}[A] }{2\gamma_B(\gamma_B+\gamma_A)}
  \end{split}
  \end{equation}
where $\sigma_0$ and $\sigma_1$ are given in Definition \ref{def.s0s1}.
\end{prop}
\begin{proof}
	See Section \ref{pf.VarianceBoundNonLinearProcess}.
\end{proof}
\begin{remark}\label{rmk.Varboundcoeff}
{{It is shown in the derivation of Proposition \ref{prop.VarianceBoundNonLinearProcess}}
that $\sigma_0$ and $\sigma_1$ 
{in} \eqref{eq.VarianceDownstreamProcessInequality} are the first two coefficients of an appropriate 
{formulation} of the Newton series {expansion} of $R(A)$. {In particular,} 
$R(A) = \sum_{n=0}^{\infty}\sigma_n\psi_n(A)$ where the functions $\psi_n(A)$, {defined in Proposition~\ref{prop.SeriesExpansion}, are a variant of the Poisson-Charlier polynomials \cite{ozmen2015}}.} 
\end{remark}
{
\begin{remark}
As follows from Proposition \ref{prop.VarianceBoundNonLinearProcess}, $\sigma_0$ and $\sigma_1$, \emph{i.e.} the first two coefficient{s} in the Newton series expansion of $R(A)$ lead to a lower bound for the variance of $B$, as this is quantified in the Proposition.
This is in contrast to \ac{LNA} approaches whereby $R(A)$ is approximated with the first two terms of its Taylor series expansion. The latter leads only {to} an approximate expression for the variance that is not a bound. Furthermore, numerical 
{computations}
illustrate that the bound in Proposition \ref{prop.VarianceBoundNonLinearProcess} is also very close to the actual variance even when $R(A)$ is highly nonlinear, a regime where \ac{LNA} approaches can be very inaccurate. Such examples will be provided in Section~\ref{sec.Examples}.
\end{remark}}
\begin{remark}\label{rmk.Prop3LineaEquality}
If $R(A)$ is a linear function then equation \eqref{eq.VarianceDownstreamProcessInequality} holds with equality. This is stated in Proposition \ref{prop.VarianceBoundLinearProcess} {below}.
\end{remark}

\begin{prop}\label{prop.VarianceBoundLinearProcess}
	Consider the system in \eqref{eq.ChemicalMasterEquationProcessStationary} {
{with Assumption \ref{assum.finiteMQB} satisfied}}. If $R(A)=R_cA$, {were $R_c \in \mathbb{R_>}$ is a constant,} {then} the bound for the variance in Proposition \ref{prop.VarianceBoundNonLinearProcess} holds with equality and is given by
    \begin{equation}\label{eq.EqualityVarianceLinearProcess}
      var(B) =\frac{R_c\mathbb{E}_{\mathcal{P}}[A](\gamma_A + \gamma_B)(\mathbb{E}[Q^2] + \mathbb{E}[Q])+R_c^2\mathbb{E}_{\mathcal{P}}[A]2(\mathbb{E}[Q])^2 }{2\gamma_B(\gamma_A+\gamma_B)}
    \end{equation}
    \end{prop}
\begin{proof}
See Section \ref{pf.VarianceBoundLinearProcess}.
\end{proof}
\begin{remark}
{
{The fact that the bound in \eqref{eq.VarianceDownstreamProcessInequality} is equal in this case to the expression in \eqref{eq.EqualityVarianceLinearProcess} can be shown  by} explicitly computing $\sigma_0$ and $\sigma_1$ in Definition \ref{def.s0s1} and using them in \eqref{eq.VarianceDownstreamProcessInequality} with $R(A)= R_cA$. {The equality of this expression with the variance of $B$ is obtained by deriving
an exact expression for the variance analytically} using \ac{LNA} approaches, see \emph{e.g.} \cite{Kampen2007} and \cite{Lestas2008}.}
\end{remark}

Finally, we give a result on the covariance of $A$ and $B$.
\begin{prop}\label{prop.COVarianceBoundNonLinearProcess}
Consider the {system} in \eqref{eq.ChemicalMasterEquationProcessStationary} {
{with Assumption \ref{assum.finiteMQB} satisfied}}. Then $cov(A,B) = cov(B,A)$ satisfies
  \begin{equation}\label{eq.COVarianceBoundNonLinearProcess}
    cov(A,B) = \mathbb{E}[Q]\frac{\sigma_1\mathbb{E}_{\mathcal{P}}[A]}{\gamma_A + \gamma_B}
  \end{equation}
  where $\sigma_1$ is given in \eqref{eq.Sigma0Sigma1Statement}.
\end{prop}
\begin{proof}
See Section \ref{pf.COVarianceBoundNonLinearProcess}.
\end{proof}
\begin{remark}
Note that Proposition \ref{prop.COVarianceBoundNonLinearProcess} holds with equality also when $R(A)$ is nonlinear.
\end{remark}
\begin{remark}
{An immediate consequence of Proposition \ref{prop.COVarianceBoundNonLinearProcess} and Proposition \ref{prop.VarianceBoundNonLinearProcess} {is a bound for the  correlation coefficient $\rho_{AB}\in [-1,1]$ of $A$ and $B$ as stated below}}
\begin{equation}\label{eq.corrcoeff}
\begin{split}
	\rho_{AB}&= \frac{cov(A,B)}{\sqrt{var(A)}\sqrt{var(B)}}\\
&\geq\frac{\mathbb{E}[Q]\mathbb{E}_{\mathcal{P}}[A]\sigma_1}{(\gamma_A+\gamma_B)\sqrt{\mathbb{E}_{\mathcal{P}}[A]}\sqrt{\frac{\sigma_0(\gamma_A+\gamma_B)(\mathbb{E}[Q^2] + \mathbb{E}[Q]) + \sigma_1^22(\mathbb{E}[Q])^2 \mathbb{E}_{\mathcal{P}}[A] }{2\gamma_B(\gamma_A+\gamma_B)}}}
\end{split}
\end{equation}
\end{remark}

The results stated above have a number of useful properties relative to other more conventional approaches for quantifying the variance {by means of analytical expressions involving the reaction rates}. As mentioned in Section \ref{sec.Introduction}, Proposition \ref{prop.VarianceBoundNonLinearProcess} provides a bound for the variance when {the} transition rate $R(a)$ is nonlinear whereas existing methods {like \ac{LNA} {approaches}} give only approximate values. As will be shown in Section \ref{sec.Examples} the variance calculation through the \ac{LNA} is not as close to the variance of $B$
as is the bound provided in Proposition \ref{prop.VarianceBoundNonLinearProcess}. This is especially important in the case of highly nonlinear propensity functions where \acp{LNA} become less accurate. {Furthermore, the numerical investigation in Section \ref{sec.Examples} demonstrates that the bound maintains its accuracy despite the presence of bursts.}

{The methodology used to derive these results is also of independent interest, based on the use of the {\em Newton series} expansion and an appropriate exploitation of {\em spectral properties} of the master equation.}

\section{Proofs}\label{sec.MYAppendix}
{{In this section we} 
provide {the proofs} of the results presented in the previous {section.}}
{To facilitate readability, we give an overview first of the corresponding derivations. The detailed derivations are given in the {subsections} that follow}.\\

{{\em Overview of proof of Proposition \ref{prop.VarianceBoundNonLinearProcess}.} The proof is provided in Section \ref{pf.VarianceBoundNonLinearProcess} and makes use of an intermediate result derived in Section \ref{prop.AppendSeriesExpansion} and auxiliary algebraic results derived in Section \ref{prop.AppendixMomentsOfPsi}.}

{In particular, the proof starts by defining the Probability Generating Functions $T_a(z)$ and $M(z)$ in \eqref{eq.pgfs} associated with $\mathbb{P}(a,b)$ and $\mathbb{P}_Q(q)$, respectively;} we then express the first two moments of $B$ as a function of those and their derivatives $T'_a(1)$ and $M'(1)$.

{We then derive an expression {for} 
$T'_a(1)$ in terms 
of the eigenfunctions $\psi_n$ and {the corresponding} eigenvalues of the considered \ac{CME}.
{We show} that {$T'_a(1)$} {can be expressed by means of} a reformulation of the \emph{Newton series} {expansion} of $R(A)$ {on a basis provided by} the eigenfunctions $\psi_n$. {This therefore yields} a connection between the variance of $B$ and the \emph{Newton series} expansion of $R(A)$.}

{We then proceed to show that the resulting expression {for} the variance of $B$ in terms of {$T'_a(1)$} is the sum of non-negative {terms involving the coefficients} $\sigma_n$ appearing in the \emph{Newton} series {expansion} of $R(A)$. {By
{retaining the first two terms} we deduce a bound for $var(B)$ in terms of $\sigma_0$ and $\sigma_1$, {as stated in Proposition \ref{prop.VarianceBoundNonLinearProcess}.}}}\\

{{\em Overview of proof of Proposition \ref{prop.VarianceBoundLinearProcess}.} The proof is provided in Section \ref{pf.VarianceBoundLinearProcess}.} {
We {derive first  an expression for the bound in \eqref{eq.VarianceDownstreamProcessInequality} when the $R(A)$ is a linear function by substituting in \eqref{eq.VarianceDownstreamProcessInequality} the form of $R(A)$ under consideration.
We then make use of {known results for the moments of the \ac{CME} when the rates are linear to derive an expression for the actual variance of $B$. The latter coincides with the expression for the variance bound previously derived, thus showing that the variance bound in Proposition \ref{prop.VarianceBoundNonLinearProcess} is exact in this case.}}}\\

{{\em Overview of proof of Proposition \ref{prop.COVarianceBoundNonLinearProcess}.} The proof is provided in Section \ref{pf.COVarianceBoundNonLinearProcess}.}
{We show {that
$cov(A,B)$}
{can also be represented as a}
function of {the quantity $T'_a(1)$ that was used in 
the proof of Proposition \ref{prop.VarianceBoundNonLinearProcess}.}
By making {use of results derived in the proof} of Proposition \ref{prop.MomentsOfPsi} and Proposition \ref{prop.sigmas} we express the considered quantity as a function of the coefficient $\sigma_1$.}\\

{In Section \ref{prop.AppendixMomentsOfPsi} we { state and prove two auxiliary results that are used in the derivations and lead also to the quantities introduced in} Definition \ref{def.s0s1}. We first {present} Proposition \ref{prop.MomentsOfPsi} an intermediate result on the moments of $\psi_n(a)$. In Proposition \ref{prop.sigmas}, {we use Proposition \ref{prop.MomentsOfPsi} to show} that the coefficients $\sigma_n$ can  {be expressed as an expectation with respect to a Poisson distribution.}}

\subsection{Proof of Proposition \ref{prop.VarianceBoundNonLinearProcess}}\label{pf.VarianceBoundNonLinearProcess}
\begin{proof}[Proof of Proposition \ref{prop.VarianceBoundNonLinearProcess}]
{Consider the probability generating functions defined below, associated with $B$ and $Q$, respectively}
\begin{equation}\label{eq.pgfs}
T_a(z){:=}\sum_{b=0}^{\infty}z^b\mathbb{P}(a,b), \quad M(z){:=}\sum_{q=0}^{\infty}z^q\mathbb{P}_{Q}(q)
\end{equation}
{{with $\text{dom}(T_a) =  \text{dom}(M) {\subseteq} \{ z \in {\mathbb{R}_\geq} 
\mid |z| \leq 1 \}$.}} We consider \eqref{eq.ChemicalMasterEquationProcessStationary} and under the transformations in \cref{eq.pgfs} we {deduce} 
the following equation
\begin{equation}\label{eq.ExpansionProofS1Transformation}
\begin{split}
  \gamma_A [(a+1)T_{a+1}(z) - aT_a(z)] + F[T_{a-1}(z) - T_a(z)] \\
  + \gamma_B(1-z)T'_{a}(z) + R(a)(M(z)-1)T_a(z)= 0
\end{split}
\end{equation}
{Note that by setting} $z=1$ we {have} 
$T_{a}(1) = \mathbb{P}_A(a)$, where $A$ in the considered case has a Poisson distribution {(see \cite{Kampen2007}) as discussed in the main text.} Secondly, we {can deduce} 
moments of $B$ given that $A=a$ by differentiating \eqref{eq.ExpansionProofS1Transformation} with respect to $z$ and setting $z=1$.
In particular, we have
\begin{equation}\label{eq.HPrimeAndSecond}
 {T'_a(1) = \sum_{b=0}^{\infty}b\mathbb{P}(a,b), \quad T''_a(1)} = \sum_{b=0}^{\infty}b(b-1)\mathbb{P}(a,b)
\end{equation}
Hence by differentiating \eqref{eq.ExpansionProofS1Transformation} with respect to $z$ and then setting $z=1$ we obtain {the following {equation} involving the first moment}
\begin{equation}\label{eq.DiscussionCMEOpertorS2}
\begin{split}
  &\gamma_A [(a+1)T'_{a+1}(1) - aT'_a(1)] + F[T'_{a-1}(1) - T'_a(1)] \\
  &- \gamma_BT'_{a}(1) + R(a)(M'(1)T_a(1) )= 0
\end{split}
\end{equation}
where $M'(1) = \mathbb{E}[Q]$.

For the second moment we {have} 
{\begin{equation}\label{eq.DiscussionCMEOpertorS3}
\begin{split}
  &\gamma_A [(a+1)T''_{a+1}(1) - aT''_a(1)] + F[T''_{a-1}(1) - T''_a(1)] \\
  &- 2\gamma_B T''_{a}(1) + R(a)(M''(1)T_a(1)+2M'(1)T'_a(1))= 0
\end{split}
\end{equation}}
where $M''(1) = \mathbb{E}[Q^2]-\mathbb{E}[Q]$.

{
{By summing \eqref{eq.DiscussionCMEOpertorS2} over all $a$  and using the property $\lim_{a\to\infty}a\sum_{b=0}^\infty b\mathbb{P}(a,b)<\infty$, which follows from Assumption \ref{assum.finiteMQB}, we get}} 
\begin{equation}\label{eq.VarianceBoundNonLinearProcessS1}
\sum_{a=0}^{\infty} T'_a(1) = \frac{M'(1)}{\gamma_B} \sum_{a=0}^{\infty} R(a)T_a(1)
\end{equation}
Note that the {left} hand side of \eqref{eq.VarianceBoundNonLinearProcessS1} is {the expectation} of $B$ {and the right hand side {is a scaled version (by a factor of $M'(1)/\gamma_B$) of the expectation of $R(A)$}} with respect to the known {distribution of $A$ since $T_a(1) = \mathbb{P}_A(a)$}.
{Hence 
{\eqref{eq.VarianceBoundNonLinearProcessS1} can be written as}}
\begin{equation}\label{eq.VarianceBoundNonLinearProcessS2}
      \mathbb{E}[B] = \frac{\mathbb{E}[Q]}{\gamma_B}\mathbb{E}_{\mathcal{P}}[R(A)]
\end{equation}
For the second moment of $B$
{we proceed analogously with the first moment by summing \eqref{eq.DiscussionCMEOpertorS3} over all {$a$} and using the property $\lim_{a\to\infty}a\sum_{b=0}^\infty b^2\mathbb{P}(a,b)<\infty$, which follows from Assumption \ref{assum.finiteMQB}}, and obtain 
{\begin{equation}\label{eq.VarianceBoundNonLinearProcessS3}
      \sum_{a=0}^{\infty} T''_a(1) = \frac{1}{2\gamma_B}\left(M''(1)\sum_{a=0}^{\infty} R(a)T_a(1) + 2M'(1)\sum_{a=0}^{\infty} R(a)T'_a(1)\right)
\end{equation}}
{From \eqref{eq.HPrimeAndSecond} and \eqref{eq.VarianceBoundNonLinearProcessS3}} we obtain
\begin{equation}\label{eq.VarianceBoundNonLinearProcessS4}
      \mathbb{E}[B^2]-\mathbb{E}[B]  =  \frac{1}{2\gamma_B}\left(M''(1)\mathbb{E}_{\mathcal{P}}[R(A)]+2M'(1)\mathbb{E}[BR(A)]\right)
\end{equation}
where for the last term on the right hand side of \eqref{eq.VarianceBoundNonLinearProcessS3} we have applied the definition of {$T'_a(1)$} {in \eqref{eq.HPrimeAndSecond}}.\\

We now seek in the remainder {of the proof} to {derive an expression}
{for the last term in \eqref{eq.VarianceBoundNonLinearProcessS4}.}

\noindent Consider the following representation of $T'_a(1)${
\begin{equation}\label{eq.DiscussionCMEOpertorS4}
  T'_a(1) = \frac{T_a(1)G(a)}{\gamma_B}
\end{equation}}
{We seek to determine a representation of the function $G(a)$ \eqref{eq.DiscussionCMEOpertorS4}} {in terms of the known probability distribution of $A$ and the moments of $Q$}. {We} {substitute} \eqref{eq.DiscussionCMEOpertorS4} into \eqref{eq.DiscussionCMEOpertorS2} and obtain
\begin{multline}\label{eq.DiscussionCMEOpertorSextrStep1}
\frac{\gamma_A}{\gamma_B}[(a+1)T_{a+1}(1)G(a+1) - {aT_{a}(1)G(a)]}\\
+\frac{F}{\gamma_B}[T_{a-1}(1)G(a-1)- T_{{a}}(1){G(a)]}\\
-{G(a)T_a(1)+M'(1)R(a)T_a(1)}=0
\end{multline}
Note that from the definition of {the} Poisson distribution we have
\begin{equation}\label{eq.DiscussionCMEOpertorSextrStep2}
T_{a+1}(1) = \frac{F}{\gamma_A(a+1)}T_a(1), \qquad T_{a-1}(1) = \frac{a\gamma_A}{F}T_a(1)
\end{equation}
By {rearranging} \eqref{eq.DiscussionCMEOpertorSextrStep1} and using \eqref{eq.DiscussionCMEOpertorSextrStep2} leads to
\begin{equation}\label{eq.DiscussionCMEOpertorS5}
T_a(1)[(\mathbb{S}-\mathbb{I})G(a)+M'(1)R(a)]=0
\end{equation}
where $\mathbb{S}$, defined in \eqref{eq.OperatorS}, is an operator acting upon $G(a)$ {and is given by}
\begin{equation}\label{eq.OperatorS}
  \mathbb{S}=\gamma_B^{-1}(\gamma_Aa\Delta_{-1} +F \Delta_1) = \frac{\gamma_A}{\gamma_B}(a\Delta_{-1} +\mathbb{E}_{\mathcal{P}}[A] \Delta_1)
\end{equation}
By making use of \eqref{eq.DiscussionCMEOpertorS5} we find that {$G(a)$ satisfies} 
    \begin{equation}\label{eq.VarianceBoundNonLinearProcessS5}
        (\mathbb{I}-\mathbb{S})G(a) = M'(1)R(a)
    \end{equation}
where $\mathbb{I}$ is the identity operator.

{We now seek to exploit spectral properties of $\mathbb{S}$ in order to derive an expression for $G(a)$ and hence of $T'_a(1)$. In particular we will derive an expression for $G(a)$ in terms of the eigenfunctions and eigenvalues of operator $\mathbb{S}$. We {derive those below}, {following the approach in \cite{Roman1984}}.}

\noindent Let us note that $\Delta_1[(a)_k] = k(a)_{k-1}$ and $a\Delta_{-1}[(a)_k] = -k(a)_k$ {and hence} we can write $\mathbb{S}[(a)_k]$ {as} 
\begin{equation}\label{eq.DiscussionCMEOpertorS7}
  \mathbb{S}[(a)_k] = -\frac{k\gamma_A}{\gamma_B}[(a)_k-\mathbb{E}_{\mathcal{P}}[A](a)_{k-1} ]
\end{equation}

To facilitate the analysis of \eqref{eq.DiscussionCMEOpertorS7} we introduce {an operator $\mathbb{F}$ that is acting on
{polynomials in}
variable $a$, with real coefficients, with $a$ taking values in $\mathbb{Z}_{\geq}$. In particular, operator $\mathbb{F}$ is defined as follows. For a monomial $a^k$, $a,k\in\mathbb{Z}_{\geq}$ we have} 
 $\mathbb{F}[a^k]=(a)_k$. {Furthermore, $\mathbb{F}$ is defined to be} linear, \emph{i.e.} $\mathbb{F}[a_1^{k_1}+a_2^{k_2}]=(a_1)_{k_1}+(a_2)_{k_2}$, $\mathbb{F}[\lambda a^k]=\lambda (a)_k$, $\lambda\in\mathbb{R}$.

{We can hence} rewrite \eqref{eq.DiscussionCMEOpertorS7} as
\begin{equation}\label{eq.DiscussionCMEOpertorExtraStep1}
  \mathbb{S}[\mathbb{F}[a^k]] = -\frac{k\gamma_A}{\gamma_B} \mathbb{F}[a^{k-1}(a-\mathbb{E}_{\mathcal{P}}[A]) ]
\end{equation}
{which can also be written as}
\begin{equation}\label{eq.DiscussionCMEOpertorExtraStep2}
  \mathbb{S}[\mathbb{F}[a^k]] = -\frac{\gamma_A}{\gamma_B} \mathbb{F}\left[\dfrac{d}{da}[a^{k}](a-\mathbb{E}_{\mathcal{P}}[A]) \right]
\end{equation}
{We now apply {the composition of operator $\mathbb{S}$ with operator $\mathbb{F}$ to}} $(a-\mathbb{E}_{\mathcal{P}}[A])^k$ which yields
\begin{equation}\label{eq.DiscussionCMEOpertorExtraStep3}
  \mathbb{S}[\mathbb{F}[(a-\mathbb{E}_{\mathcal{P}}[A])^k]] = -\frac{\gamma_A}{\gamma_B}\mathbb{F}\left[\dfrac{d}{da}[(a-\mathbb{E}_{\mathcal{P}}[A])^k](a-\mathbb{E}_{\mathcal{P}}[A]) \right]
\end{equation}
{By observing that $\dfrac{d}{da}[(a-\mathbb{E}_{\mathcal{P}}[A])^k]=k(a-\mathbb{E}_{\mathcal{P}}[A])^{k-1}$ we have}
\begin{equation}\label{eq.DiscussionCMEOpertorExtraStep4}
 \mathbb{S}[\mathbb{F}[(a-\mathbb{E}_{\mathcal{P}}[A])^k]] = -k\frac{\gamma_A}{\gamma_B} \mathbb{F}\left[(a-\mathbb{E}_{\mathcal{P}}[A])^{k} \right]
\end{equation}
{We now} use the binomial expansion to rewrite the eigenfunctions {and corresponding eigenvalues, which take the form}
\begin{equation}\label{eq.DiscussionCMEOpertorS9}
  \lambda_n =-n \frac{\gamma_A}{\gamma_B}
\end{equation}
\begin{equation}\label{eq.DiscussionCMEOpertorS10}
  \psi_n(a) = \sum_{k=0}^{n}\binom{n}{k}( -\mathbb{E}_{\mathcal{P}}[A])^k(a)_{n-k}
\end{equation}

{Using the expressions derived for the eigenfunctions} and eigenvalues of $\mathbb{S}$ {we now} proceed 
{to derive an expression for} $G(a)$. {In particular, by noting the decomposition of $R(a)$ in terms of the eigenfunctions $\psi_n$ given in Proposition \ref{prop.SeriesExpansion} {(stated in Section~\ref{prop.AppendSeriesExpansion})} we} make the following ansatz on $G(a)$
\begin{equation}\label{eq.VarianceBoundNonLinearProcessHyperExtension0}
  G(a)= M'(1)\sum_{n=0}^{\infty} \sigma_n \zeta(n) \psi_n(a)
\end{equation}
where $\zeta(n)$ is {a function} to be determined and $\sigma_n$ are {as} defined in Proposition \ref{prop.SeriesExpansion}. 
By substituting \eqref{eq.VarianceBoundNonLinearProcessHyperExtension0} in \eqref{eq.VarianceBoundNonLinearProcessS5} we obtain
\begin{equation}\label{eq.VarianceBoundNonLinearProcessHyperExtension2}
  M'(1)\sum_{n=0}^{\infty} \sigma_n \left(1+n\frac{\gamma_A}{\gamma_B}\right) {\zeta(n)} \psi_n(a) =M'(1)R(a)
\end{equation}
Noting the series expansion of $R(a)$ in Proposition \ref{prop.SeriesExpansion} with $r=\mathbb{E}_{\mathcal{P}}[A]$ we deduce that
\begin{equation}\label{eq.VarianceBoundNonLinearProcessHyperExtension4}
  \zeta(n) = \frac{\gamma_B}{n\gamma_A+\gamma_B}
\end{equation}
Therefore from \eqref{eq.DiscussionCMEOpertorS4}, \eqref{eq.VarianceBoundNonLinearProcessHyperExtension0}, \eqref{eq.VarianceBoundNonLinearProcessHyperExtension4} we deduce the following expression for $T'_a(1)$
\begin{equation}\label{eq.VarianceBoundNonLinearProcessS7}
      T'_a(1) = \frac{T_a(1)M'(1)}{\gamma_B}\sum_{n=0}^{\infty} \sigma_n \frac{\gamma_B}{n\gamma_A+\gamma_B} \psi_n(a)
    \end{equation}

We now make use of the series expansion {of $T'_a(1)$ in \eqref{eq.VarianceBoundNonLinearProcessS7}} to compute the moments of $B$.

\noindent {We evaluate} first the unknown term in the right hand side of \eqref{eq.VarianceBoundNonLinearProcessS4}. In particular, by noting that $\mathbb{E}[BR(A)] = \sum_{a=0}^{\infty}R(a)T'_a(1)$ and from the expression of $T'_a(1)$ in \eqref{eq.VarianceBoundNonLinearProcessS7} we have
\begin{equation}\label{eq.VarianceBoundNonLinearProcessS8}
\begin{split}
  \mathbb{E}[BR(A)] &= \frac{M'(1)}{\gamma_B}\sum_{a=0}^{\infty}R(a)\sum_{n=0}^{\infty} \sigma_n \frac{\gamma_B}{n\gamma_A+\gamma_B} \psi_n(a)T_a(1)\\
  & = \frac{\mathbb{E}[Q]}{\gamma_B}\mathbb{E}_{\mathcal{P}}\left[\sum_{n'=0}^{\infty}\sum_{n=0}^{\infty}\sigma_{n'}\sigma_n \frac{\gamma_B}{n\gamma_A+\gamma_B} \psi_{n'}(A)\psi_n(A)\right]
\end{split}
\end{equation}
where $M'(1)= \mathbb{E}[Q]$ and in the second step we have used the results of Proposition~\ref{prop.SeriesExpansion}. {We now} evaluate the expectation in \eqref{eq.VarianceBoundNonLinearProcessS8} as follows
\begin{equation}\label{eq.VarianceBoundNonLinearProcessS10}
 \mathbb{E}_{\mathcal{P}}\left[\sum_{n'=0}^{\infty}\sum_{n=0}^{\infty}\sigma_n'\sigma_n \frac{\gamma_B}{n\gamma_A+\gamma_B} \psi_{n'}(A)\psi_n(A)\right]= \sum_{n=0}^{\infty}\sigma_n^2\frac{\gamma_Bn!(\mathbb{E}_{\mathcal{P}}[A])^n}{n\gamma_A+\gamma_B}
\end{equation}
where the equality holds given the results of Proposition~\ref{prop.MomentsOfPsi} {(stated in Section \ref{prop.AppendixMomentsOfPsi})}. Hence we can now write the following expression for $\mathbb{E}[BR(A)]$
\begin{equation}\label{eq.VarianceBoundNonLinearProcessS10a}
\begin{split}
 &\mathbb{E}[BR(A)]= \frac{M'(1)}{\gamma_B}\sum_{n=0}^{\infty}\sigma_n^2\frac{\gamma_Bn!(\mathbb{E}_{\mathcal{P}}[A])^n}{n\gamma_A+\gamma_B}
\end{split}
\end{equation}
We finally express $var(B)$ as in \eqref{eq.VarianceBoundNonLinearProcessS11} by making use {of \eqref{eq.VarianceBoundNonLinearProcessS2}, \eqref{eq.VarianceBoundNonLinearProcessS4},  
and \eqref{eq.VarianceBoundNonLinearProcessS10a}}
\begin{multline}\label{eq.VarianceBoundNonLinearProcessS11}
     var(B)=\frac{1}{2\gamma_B}\Big[ (\mathbb{E}[Q^2]-\mathbb{E}[Q])\mathbb{E}_{\mathcal{P}}[R(A)]+\frac{2(\mathbb{E}[Q])^2}{\gamma_B}\sum_{n=0}^{\infty}\sigma_n^2\frac{\gamma_Bn!(\mathbb{E}_{\mathcal{P}}[A])^n}{n\gamma_A+\gamma_B}\Big] \\
     +\frac{\mathbb{E}_{\mathcal{P}}[R(A)]\mathbb{E}[Q]}{\gamma_B} - \left(\frac{\mathbb{E}_{\mathcal{P}}[R(A)]\mathbb{E}[Q]}{\gamma_B}\right)^2
\end{multline}
{It should be noted} that from Proposition \ref{prop.MomentsOfPsi} we have $\mathbb{E}_{\mathcal{P}}[\psi_n(A)]=\delta_{n,0}$ and hence by computing the expectation with respect to the Poisson distribution on both sides {of \eqref{eq.FunctionDecomposition}} 
we have that $\mathbb{E}_{\mathcal{P}}[R(A)] = \sigma_0$. By substituting $\sigma_0$ in \eqref{eq.VarianceBoundNonLinearProcessS11} and through some manipulation we obtain
\begin{equation}\label{eq.VarianceBoundNonLinearProcessS12}
  var(B)= \frac{1}{2\gamma_B} \Big[ (\mathbb{E}[Q^2]-\mathbb{E}[Q])\sigma_0+ \frac{2(\mathbb{E}[Q])^2}{\gamma_B}\sum_{n=1}^{\infty}\frac{\gamma_B\sigma_n^2 n!(\mathbb{E}_{\mathcal{P}}[A])^n}{n\gamma_A+\gamma_B}\Big] + \frac{\mathbb{E}[Q]}{\gamma_B}\sigma_0
\end{equation}
Since all terms in {the summation in} \eqref{eq.VarianceBoundNonLinearProcessS12} are positive, {using only the first term
we obtain}
\begin{equation}\label{eq.VarianceBoundNonLinearProcessS13}
    var(B)> \frac{ (\gamma_A+\gamma_B)\sigma_0(\mathbb{E}[Q^2] + \mathbb{E}[Q]) + 2(\mathbb{E}[Q])^2 \sigma_1^2\mathbb{E}_{\mathcal{P}}[A] }{2\gamma_B(\gamma_B+\gamma_A)}
\end{equation}
\end{proof}

\subsection{Expansion of $R(A)$}\label{prop.AppendSeriesExpansion}
\begin{prop}\label{prop.SeriesExpansion}
Consider a function $R(A)$, $R: \mathbb{Z}_{\geq} \rightarrow \mathbb{R}_{\geq}$ and the function
\begin{equation}\label{eq.PsiFactorialCalculation}
  \psi_n(A) = \sum_{k=0}^{n}\binom{n}{k}(-r)^k(A)_{n-k}
\end{equation}
where $r \in \mathbb{R}_{>}$.
Then the following equality holds
\begin{equation}\label{eq.FunctionDecomposition}
    R(A) = \sum_{n=0}^{\infty} \sigma_n \psi_n(A)
  \end{equation}
where
\begin{equation}\label{eq.SigmaFactorialCalculation}
\sigma_n = \sum_{k={n}}^{\infty} \binom{k}{n}\left.\frac{\Delta_1^k [R](A)}{k!}\right|_{A=0} r^{k-n}
\end{equation}
\end{prop}
\begin{proof}[Proof of Proposition \ref{prop.SeriesExpansion}]
{We consider the Newton series expansion of the function $R(A)$}
\begin{equation}\label{eq.ExpansionProofS1}
  R(A) = \sum_{k=0}^{\infty} \rho_k(A)_k, \qquad \rho_k = \left.\frac{\Delta_1^k[R](A)}{k!}\right|_{A=0}
\end{equation}
{It should be noted that the Newton series in \eqref{eq.ExpansionProofS1} {exists} 
for all integer numbers {$A \in \mathbb{Z}_\geq$} for any function {$R:\mathbb{Z}_\geq \to \mathbb{R}_>$}, {see \cite[Ch. VII, p.358/360]{Jordan1965}, \cite{Koenig2021}, \cite{Spiegel1994}}.}

\noindent We now make use of the fact that the falling factorial can be expressed in terms of the functions $\psi_n(A)$. In particular, the following relation holds which can be verified by substituting $\psi_n(A)$ in the right hand side of \eqref{eq.ExpansionProofS2}.
\begin{equation}\label{eq.ExpansionProofS2}
  (A)_k = \sum_{n=0}^{k} \binom{k}{n}r^{k-n}\psi_n(A)
\end{equation}
Substituting \eqref{eq.ExpansionProofS2} into \eqref{eq.ExpansionProofS1} yields
\begin{equation}\label{eq.ExpansionProofS3}
\begin{split}
  &R(A) = \sum_{n=0}^{\infty} \sigma_n \psi_n(A)\\
  &\sigma_n = \sum_{k={n}}^{\infty} \binom{k}{n}\left.\frac{\Delta_1^k [R](A)}{k!} \right|_{A=0}r^{k-n}
\end{split}
\end{equation}
\end{proof}
\subsection{Proof of Proposition \ref{prop.VarianceBoundLinearProcess}}\label{pf.VarianceBoundLinearProcess}
\begin{proof}[Proof of Proposition \ref{prop.VarianceBoundLinearProcess}]
When $R(A)=R_cA$ the series \eqref{eq.FunctionDecomposition}, has $\sigma_n = 0, \forall n \geq 2$. From \eqref{eq.SigmaFactorialCalculation} and \eqref{eq.PsiFactorialCalculation}, by setting $r = \mathbb{E}_{\mathcal{P}}[A]$, we compute $\sigma_0$, $\sigma_1$ and $\psi_0(A), \psi_1(A)$ as follows
\begin{equation}\label{VarianceBoundLinearProcessProofStep1}
\begin{split}
  \sigma_0 &= R_c\mathbb{E_\mathcal{P}}[A] \\
  \sigma_1 &= R_c \\
  \psi_0(A) &= 1 \\
  \psi_1(A) &= A - \mathbb{E_\mathcal{P}}[A]
\end{split}
\end{equation}
By substituting the expressions above in {the right-hand side (RHS) of} \eqref{eq.VarianceDownstreamProcessInequality} we obtain {the RHS of} \eqref{eq.EqualityVarianceLinearProcess}.

{The} covariance matrix of the vector of random variables $[A,B]^T$ 
{for a linear {$R(A)=R_c A$}} satisfies {at steady state} {a Lyapunov equation {\cite{Kampen2007}} of the form} \eqref{eq.EqualityVarianceLinearProcessS1} 
\begin{equation}\label{eq.EqualityVarianceLinearProcessS1}
  {P \Sigma + \Sigma P^{T} + D=0}
\end{equation}
{The precise definition {of matrices $P, D$ can be found in\footnote{Note that while bursts are not explicitly taken into account {in these works} they may be easily incorporated.}} \cite{Kampen2007} and \cite{Lestas2008}.} {In particular, for} the case under consideration {the matrices $P$ and $D$ are defined as follows}
{
\begin{equation}\label{eq.MatricesPAndD}
\begin{split}
P =
  \begin{bmatrix}
    -\gamma_A & 0 \\
    \mathbb{E}[Q]R_c & -\gamma_B
  \end{bmatrix},\quad
D =
  \begin{bmatrix}
    F+\gamma_A\mathbb{E}_{\mathcal{P}}[A] & 0 \\
    0 & \mathbb{E}[Q^2]R_c \mathbb{E}_{\mathcal{P}}[A] + \gamma_B\mathbb{E}[B]
  \end{bmatrix}
\end{split}
\end{equation}
}
{and it follows that the solution to \eqref{eq.EqualityVarianceLinearProcessS1} is}
\begin{equation}\label{eq.EqualityVarianceLinearProcessS4}
\Sigma = \begin{bmatrix}
  \mathbb{E}_{\mathcal{P}}[A]                                & \mathbb{E}[Q]\frac{R_c\mathbb{E}_{\mathcal{P}}[A]}{(\gamma_A+\gamma_B)} \\
  \mathbb{E}[Q]\frac{R_c\mathbb{E}_{\mathcal{P}}[A]}{(\gamma_A+\gamma_B)} &\frac{R_c\mathbb{E}_{\mathcal{P}}[A][R_c2(\mathbb{E}[Q])^2+(\gamma_A + \gamma_B)(\mathbb{E}[Q^2] + \mathbb{E}[Q])]}{2\gamma_B(\gamma_A+\gamma_B)}
\end{bmatrix}
\end{equation}
The element $\Sigma_{2,2}=\frac{R_c\mathbb{E}_{\mathcal{P}}[A][R_c2(\mathbb{E}[Q])^2+(\gamma_A + \gamma_B)(\mathbb{E}[Q^2] + \mathbb{E}[Q])]}{2\gamma_B(\gamma_A+\gamma_B)}$ is the variance of $B$ which {proves the statement {in Proposition \ref{pf.VarianceBoundLinearProcess}}}. 
\end{proof}

\subsection{Proof of Proposition \ref{prop.COVarianceBoundNonLinearProcess}}\label{pf.COVarianceBoundNonLinearProcess}
\begin{proof}[Proof of Proposition \ref{prop.COVarianceBoundNonLinearProcess}]
{We {use first an} 
first an expression for $\mathbb{E}[AB]$ involving $T'_a(1)$ {derived in Section \ref{pf.VarianceBoundNonLinearProcess} and write this by means of $\mathbb{E}_{\mathcal{P}}[ \cdot ]$, where $\mathcal{P}$ is the Poisson distribution in \eqref{def:A_Poisson}.}}
{{In particular, note that}} $\mathbb{E}[AB] = \sum_{a=0}^{\infty}aT'_a(1)$. Hence from the expression of $T'_a(1)$ in \eqref{eq.VarianceBoundNonLinearProcessS7} we have
\begin{equation}\label{eq.COVarianceBoundNonLinearProcessProofStep1}
  \mathbb{E}[AB] = \frac{\mathbb{E}[Q]}{\gamma_B} \mathbb{E}_{\mathcal{P}}\left[A\sum_{n=0}^{\infty} \sigma_n \frac{\gamma_B}{n\gamma_A+\gamma_B} \psi_n(A)\right]
\end{equation}
By noting that $A = \mathbb{E}_{\mathcal{P}}[A] \psi_0(A) + \psi_1(A)$
we have
\begin{equation}\label{eq.COVarianceBoundNonLinearProcessProofStep2}
\begin{split}
&	\mathbb{E}_{\mathcal{P}}\left[A\sum_{n=0}^{\infty} \sigma_n \frac{\gamma_B}{n\gamma_A+\gamma_B} \psi_n(A)\right]=\\
  &= \sum_{n=0}^{\infty} \sigma_n \bigg[ \mathbb{E}_{\mathcal{P}}[A] \mathbb{E}_{\mathcal{P}}[\psi_0(A)\psi_n(A)] + \frac{\gamma_B}{\gamma_A+\gamma_B}\mathbb{E}_{\mathcal{P}}[\psi_1(A)\psi_n(A)] \bigg]\\
  &= \sigma_0 \mathbb{E}_{\mathcal{P}}[A]+\sigma_1 \frac{\gamma_B}{\gamma_A+\gamma_B}\mathbb{E}_{\mathcal{P}}[A]
\end{split}
\end{equation}
where the last step follows from Proposition \ref{prop.MomentsOfPsi} {(stated in Section \ref{prop.AppendixMomentsOfPsi}).}
From the definition of covariance we have
\begin{equation}\label{eq.COVarianceBoundNonLinearProcessProofStep3}
\begin{split}
cov(A,B)&=\mathbb{E}[AB]-\mathbb{E}[A]\mathbb{E}[B]  \\
&= \frac{\mathbb{E}[Q]}{\gamma_B}\Big[ \sigma_0 \mathbb{E}_{\mathcal{P}}[A] +\sigma_1 \frac{\gamma_B}{\gamma_A+\gamma_B}\mathbb{E}_{\mathcal{P}}[A] - \mathbb{E}_{\mathcal{P}}[A] \mathbb{E}_{\mathcal{P}}[R(A)]\Big]\\
&= \mathbb{E}[Q]\frac{\sigma_1\mathbb{E}_{\mathcal{P}}[A]}{\gamma_B+\gamma_A}
\end{split}
\end{equation}
where the last {equality} {follows from the relation $\sigma_0=\mathbb{E}_{\mathcal{P}}[R(A)]$}.
\end{proof}

\subsection{Properties of $\psi_n$ and $\sigma_n$}\label{prop.AppendixMomentsOfPsi}
The proof of the following proposition is based {on ideas} in \cite{Ogura1972}.
\begin{prop}\label{prop.MomentsOfPsi}
Consider function $\psi_n(A)$ as in Proposition \ref{prop.SeriesExpansion} with \\ $r=\mathbb{E}_{\mathcal{P}}[A]$, where $A$ is a random variable with Poisson distribution $\mathcal{P}$. Then the following equalities hold
\begin{equation}\label{eq.MomentOfPsiProof}
  \mathbb{E}_{\mathcal{P}}[\psi_n(A)] = \delta_{n,0}, \quad \mathbb{E}_{\mathcal{P}}[\psi_n(A)\psi_{n'}(A)] = n!(\mathbb{E}_{\mathcal{P}}[A])^n\delta_{n',n}
\end{equation}
\end{prop}
\begin{proof}[Proof of Proposition \ref{prop.MomentsOfPsi}]
The first equality can be derived from the definition of $\psi_n(A)$ in Proposition \ref{prop.SeriesExpansion} with $r=\mathbb{E}_{\mathcal{P}}[A]$, and the fact that for a random variable that has a Poisson distribution it holds that  $\mathbb{E}_{\mathcal{P}}[(A)_k] = {(\mathbb{E}_{\mathcal{P}}[A])^k}$.
By taking the expectation on both sides of \eqref{eq.PsiFactorialCalculation} with respect to the Poisson distribution one obtains
\begin{equation}
   \mathbb{E}_{\mathcal{P}}[\psi_n(A)] = \delta_{n,0}
\end{equation}
In order to prove the second equality we make use of the \emph{Chu-Vandermonde} identity \cite{Roman1984}
\begin{equation}\label{eq.MomentOfPsiProof1}
	(x+y)_m = \sum_{{n}=0}^{m}\binom{m}{{n}}(x)_{m-{n}}(y)_{n}
\end{equation}
{For any $m'\in \mathbb{Z}_\geq$ with $x = A-m'$ and $y=m'$ equation \eqref{eq.MomentOfPsiProof1}} yields
\begin{equation}\label{eq.MomentOfPsiProof2}
\begin{split}
	(A)_m &=\sum_{{n}=0}^{m}\binom{m}{{n}}(A-m')_{m-{n}}(m')_{n} \\
	      &=\sum_{{n}=0}^{m}{n}!\binom{m}{{n}}\binom{m'}{{n}}(A-m')_{m-{n}}
\end{split}
\end{equation}
where the second equation is obtained from the fact that $(m)_{n}={n}!\binom{m}{{n}}$. Multiplying by $(A)_{m'}$ we obtain
\begin{equation}\label{eq.MomentOfPsiProof3}
\begin{split}
  (A)_{m'}(A)_m &= \sum_{n=0}^{m}n!\binom{m}{n}\binom{m'}{n}(A)_{m'}(A-m')_{m-n} \\
			     &= \sum_{n=0}^{m}{n}!\binom{m}{n}\binom{m'}{n}(A)_{m+m'-{n}}
\end{split}
\end{equation}
where we have made use {that for }falling factorials $(A)_{i+j} = (A)_i(A-i)_j$ {\cite{Knuth2021}}. Taking {expectations} on both sides of {\eqref{eq.MomentOfPsiProof3}} leads to the following equality
\begin{equation}\label{eq.MomentOfPsiProof4}
	\mathbb{E}_{\mathcal{P}}[(A)_{m'}(A)_{m}] = \sum_{n=0}^{m}{n}!\binom{m}{{n}}\binom{m'}{{n}} (\mathbb{E}_{\mathcal{P}}[A])^{m+m'-{n}}
\end{equation}
{Another expression for the {left-hand side of \eqref{eq.MomentOfPsiProof4}}} 
can be obtained by plugging the expression of the falling factorial from \eqref{eq.ExpansionProofS2} in both $(A)_m$ and $(A)_{m'}$, with $r = \mathbb{E}_{\mathcal{P}}[A]$. This leads to the following equality
\begin{equation}\label{eq.MomentOfPsiProof5}
  \mathbb{E}_{\mathcal{P}}[(A)_{m'}(A)_{m}] =
  {\sum_{n=0}^{m}\sum_{n'=0}^{m'}\binom{m'}{n'}\binom{m}{n} (\mathbb{E}_{\mathcal{P}}[A])^{(m+m'-n-n')}\mathbb{E}_{\mathcal{P}}[\psi_{n}(A)\psi_{n'}(A)]}
\end{equation}
{By equating 
\eqref{eq.MomentOfPsiProof5} and \eqref{eq.MomentOfPsiProof4} and by expanding the summations} we {obtain} 
\begin{equation}\label{eq.MomentOfPsiProof6}
  \mathbb{E}_{\mathcal{P}}[\psi_n(A)\psi_{n'}(A)] = n!(\mathbb{E}_{\mathcal{P}}[A])^n\delta_{n,n'}
\end{equation}
\end{proof}

\vspace{0.5cm}
{We now give the following result {that provides an interpretation of the coefficients $\sigma_n$ in \eqref{eq.SigmaFactorialCalculation} as expectations with respect to a Poisson distribution.}
{The expressions for $\sigma_0, \sigma_1$ in Definition \ref{def.s0s1}} are an immediate consequence of the following Proposition.}
\begin{prop}\label{prop.sigmas}
{Consider {function $R(.)$, coefficients $\sigma_n$, and functions $\psi_n(.)$ as in Proposition
\ref{prop.SeriesExpansion} with} $r=\mathbb{E}_{\mathcal{P}}[A]$, {{where $A$ is a random variable with Poisson distribution $\mathcal{P}$.}
Then the following relation holds}}
\begin{equation}\label{eq.ComputationOfSigmas}
	{\sigma_n= \frac{\mathbb{E}_{\mathcal{P}}[\psi_n(A)R(A)]}{n!(\mathbb{E}_{\mathcal{P}}[A])^n}}
\end{equation}
\end{prop}
\begin{proof}[Proof of Proposition \ref{prop.sigmas}]
{Consider the expansion in \eqref{eq.FunctionDecomposition}, and let us multiply both sides by $\psi_n(A)$ yielding
\begin{equation}
    \psi_n(A)R(A) = \sum_{n'=0}^{\infty} \sigma_n' \psi_n'(A)\psi_n(A)
\end{equation}
{We now take {expectations}} with respect to the Poisson distribution on both sides
\begin{equation}
	\mathbb{E}_{\mathcal{P}}[\psi_n(A)R(A)] = \sum_{n'=0}^{\infty} \sigma_n' \mathbb{E}_{\mathcal{P}}[\psi_n'(A)\psi_n(A)]
\end{equation}
{Using} Proposition \ref{prop.AppendixMomentsOfPsi} it follows that
\begin{equation}
	\mathbb{E}_{\mathcal{P}}[\psi_n(A)R(A)] = \sigma_n n!(\mathbb{E}_{\mathcal{P}}[A])^n
\end{equation}
{and from} the previous relation we obtain
\begin{equation}
\sigma_n= \frac{\mathbb{E}_{\mathcal{P}}[\psi_n(A)R(A)]}{n!(\mathbb{E}_{\mathcal{P}}[A])^n}
\end{equation}}
\end{proof}

\section{Numerical examples}\label{sec.Examples}
We evaluate the conservativeness of the bound for the variance given in \eqref{eq.VarianceDownstreamProcessInequality} by numerically computing how close this is to the actual variance. We make use of the Stochastic Simulation Algorithm software GillesPy2 available in \cite{Abel2016}.

We select as function $R(A)$ a Hill function that is commonly found in several biological processes
\begin{equation}\label{eq.RHill}
R(a)=\frac{(a/A_0)^{n_h}}{1+(a/A_0)^{n_h}}, n_h \in \mathbb{R}_{>}, A_0 \in \mathbb{R}_{>}
\end{equation}
In Fig. \ref{fig.Fig2} we compute $var(B)$, the bound in \eqref{eq.VarianceDownstreamProcessInequality}, and the \ac{LNA} for \eqref{eq.ChemicalMasterEquationProcessStationary} for the considered $R(a)$ with $n_h= 9$ and $A_0=100$ and $F=1$.
{We consider four different distributions for $Q$ corresponding to the binomial and uniform distributions and 
truncated and shifted versions {of Poisson} and geometric burst distributions \cite{Beentjes2020} {as defined} below. In particular, we consider $Q$ to take values {in the} set {$M=\{1,\ldots,m_N\}\subset \mathbb{Z}_{>}$} and the following probability mass function
{is used as an approximate Poisson distribution for $Q$ {with parameter $\lambda\in\mathbb{R}_>$}}
\begin{equation}\label{eq.approxPoiss}
\mathbb{P}_Q(q)=\frac{\frac{\lambda^{q-1}e^{-\lambda}}{(q-1)!}}{\sum_{k=1}^{m_N}\frac{\lambda^{k-1}e^{-\lambda}}{(k-1)!}}, \forall q \in M
\end{equation}}
{We also consider the following truncated version of the geometric distribution $\mathcal{G}(p)$ with parameter $p$}
\begin{equation}\label{eq.approxGeom}
\mathbb{P}_Q(q)=\frac{p(1-p)^{q-1}}{\sum_{k=1}^{m_N}p(1-p)^{k-1}}, \forall q \in M
\end{equation}

It should be noted that the Hill function \eqref{eq.RHill} is bounded yielding that $\mathbb{E}[R(A)]$ is {finite,}
and also the considered burst distributions have finite support. Hence Remark~\ref{rmk.finitesupport} applies and Assumption \ref{assum.Stationaity} holds. 

The diagrams show that \eqref{eq.VarianceDownstreamProcessInequality} is not conservative regardless of $Q$ and it is very close to $var(B)$ for {a large range} of {values of} $var(A) = \mathbb{E}_{\mathcal{P}}[A]$. The figures display also a comparison of the bound with the variance of $B$ computed through the \ac{LNA} of \eqref{eq.ChemicalMasterEquationProcessStationary}. Note that the \ac{LNA} and the bound in \eqref{eq.VarianceDownstreamProcessInequality} are accurate when $A$ is predominantly in the linear regime of $R(a)$, \emph{i.e.} for high and low values of $A$.

\noindent {The \ac{LNA} performs however particularly poorly for $\mathbb{E}_{\mathcal{P}}[A]={var(A)}\approx A_0$}, \emph{i.e.} $A$ is with high probability in the nonlinear {regime of} $R(A)$.
{The differences} are due to the effect of the coefficients $\sigma_0$ and $\sigma_1$ of the reformulated Newton Expansion which account more precisely for the highly nonlinear regimes of $R(A)$.
\begin{figure}[!htbp]
\centering
\begin{tikzpicture}
\begin{groupplot}[
    group style             = {group size = 2 by 2, vertical sep = 1cm, horizontal sep = 0.6cm, xlabels at = edge bottom , ylabels at = edge left },
    view                    = {0}{90},
    width		    = 0.48\columnwidth,
    height		    = 0.48\columnwidth,
    ylabel 	            = {$var(B)$},
    xlabel 	            = {$\mathbb{E}_{\mathcal{P}}[A]=var(A)$},
    xmin                    = 0.1,
    xmax                    = 200,
    group/xlabels at 	     = edge bottom,
    group/ylabels at  	     = edge left,
    ]
\nextgroupplot[
		xmajorgrids,
		grid style 				= {thin, dashed, black!20},
                xmajorticks=false, 
		]
\addplot
		[
            		smooth,
			draw			= red,
			mark			= none,
			line width		= 0.03cm,
            style          = solid
		]
		table
		[
			x				= MeanA,
			y				= BoundVarB,
		]
        {Bound_SSA_MeanAvsVarBAndBoundBursting_Uniform.txt};\label{plt.BoundBurstUniform}
		\addplot
		[
        	        smooth,
			draw			= black,
			mark			= none,
			line width		= 0.03cm,
		        style          = dashed
		]
		table
		[
			x				= MeanA,
			y				= VarB,
		]
		{VarB_Uniform.txt};\label{plt.VarBurstUniform}
		\addplot
		[
            		smooth,
			draw			= blue,
			mark			= none,
			line width		= 0.03cm,
            		style          = dash dot
		]
		table
		[
			x				= MeanA,
			y				= VarFromLNA,
		]
		{LNA_Bursting_Uniform.txt};\label{plt.LNAVarBurstUniform}

\nextgroupplot[
                xmajorticks=false, 
		xmajorgrids,
		grid style 				= {thin, dashed, black!20},
		yticklabel pos=right
		]
\addplot
		[
	            smooth,
			draw			= red,
			mark			= none,
			line width		= 0.03cm,
            style          = solid
		]
		table
		[
			x				= MeanA,
			y				= BoundVarB,
		]
        {Bound_SSA_MeanAvsVarBAndBoundBursting_Geometric.txt};\label{plt.BoundBurstGeometric}
		\addplot
		[
            smooth,
			draw			= black,
			mark			= none,
			line width		= 0.03cm,
            style          = dashed
		]
		table
		[
			x				= MeanA,
			y				= VarB,
		]
		{VarB_Geometric.txt};\label{plt.VarBurstGeomtric}
		\addplot
		[
            		smooth,
			draw			= blue,
			mark			= none,
			line width		= 0.03cm,
        	        style          = dash dot
		]
		table
		[
			x				= MeanA,
			y				= VarFromLNA,
		]
		{LNA_Bursting_Geometric.txt};\label{plt.LNAVarBurstGeometrix}
\nextgroupplot[
		xmajorgrids,
		grid style 				= {thin, dashed, black!20},
		]
\addplot
		[
            		smooth,
			draw			= red,
			mark			= none,
			line width		= 0.03cm,
            		style          = solid
		]
		table
		[
			x				= MeanA,
			y				= BoundVarB,
		]
        {Bound_SSA_MeanAvsVarBAndBoundBursting_Poisson.txt};\label{plt.BoundBurstPoisson}
		\addplot
		[
            smooth,
			draw			= black,
			mark			= none,
			line width		= 0.03cm,
            style          = dashed
		]
		table
		[
			x				= MeanA,
			y				= VarB,
		]
		{VarB_Poisson.txt};\label{plt.VarBurstPoisson}
		\addplot
		[
            smooth,
			draw			= blue,
			mark			= none,
			line width		= 0.03cm,
            style          = dash dot
		]
		table
		[
			x				= MeanA,
			y				= VarFromLNA,
		]
		{LNA_Bursting_Poisson.txt};\label{plt.LNAVarBurstPoisson}
\nextgroupplot[
		yticklabel pos=right,
		xmajorgrids,
		grid style 				= {thin, dashed, black!20},
	    ]
      \addplot
		[
            		smooth,
			draw			= red,
			mark			= none,
			line width		= 0.03cm,
            		style          = solid
		]
		table
		[
			x				= MeanA,
			y				= BoundVarB,
		]
        {Bound_SSA_MeanAvsVarBAndBoundBursting_Binomial.txt};\label{plt.BoundBurst_Binomial}
		\addplot
		[
            smooth,
			draw			= black,
			mark			= none,
			line width		= 0.03cm,
            style          = dashed
		]
		table
		[
			x				= MeanA,
			y				= VarB,
		]
		{VarB_Binomial.txt};\label{plt.VarBurst_Binomial}
		\addplot
		[
            smooth,
			draw			= blue,
			mark			= none,
			line width		= 0.03cm,
            style          = dash dot
		]
		table
		[
			x				= MeanA,
			y				= VarFromLNA,
		]
		{LNA_Bursting_Binomial.txt};\label{plt.LNAVarBurst_Binomial}
\end{groupplot}
\node[text width=6cm, align=center, anchor = south] at (group c1r1.north) {\subcaption{$Q\sim\mathcal{U}\{1,16\}$\label{fig_unif}}};
\node[text width=6cm, align=center, anchor = south] at (group c2r1.north) {\subcaption{$Q\sim\mathcal{G}(0.5)$\label{fig_geom}}};
\node[text width=6cm, align=center, anchor = south] at (group c1r2.north) {\subcaption{$Q\sim\mathcal{P}(8)$\label{fig_poisson}}};
\node[text width=6cm, align=center, anchor = south] at (group c2r2.north) {\subcaption{$Q\sim\mathcal{B}(20,0.4)$\label{fig_bionmial}}};
\end{tikzpicture}
\caption[test]{Numerical simulation showing the bound obtained from \eqref{eq.VarianceDownstreamProcessInequality} and the computed variance of $B$ for varying $\mathbb{E}_{\mathcal{P}}[A] = F/\gamma_A$ with $F=1, \gamma_B=10^{-2}$ and for different distributions of $Q$ {($\mathcal{P}(\lambda)$, $\mathcal{G}(p)$ refer to the distributions in \eqref{eq.approxPoiss}, \eqref{eq.approxGeom} respectively with $m_N>20$)}. $R(a)$ is a Hill function with parameters $n_h= 9$ and $A_0=100$. (\ref{plt.BoundBurst_Binomial}) displays the bound computed from \eqref{eq.VarianceDownstreamProcessInequality} and (\ref{plt.VarBurst_Binomial}) displays the variance obtained from numerical simulations of \eqref{eq.ChemicalMasterEquationProcessStationary}. The variance computed through a \ac{LNA} of system \eqref{eq.ChemicalMasterEquationProcessStationary} is displayed as (\ref{plt.LNAVarBurst_Binomial})}
\label{fig.Fig2}
\end{figure}
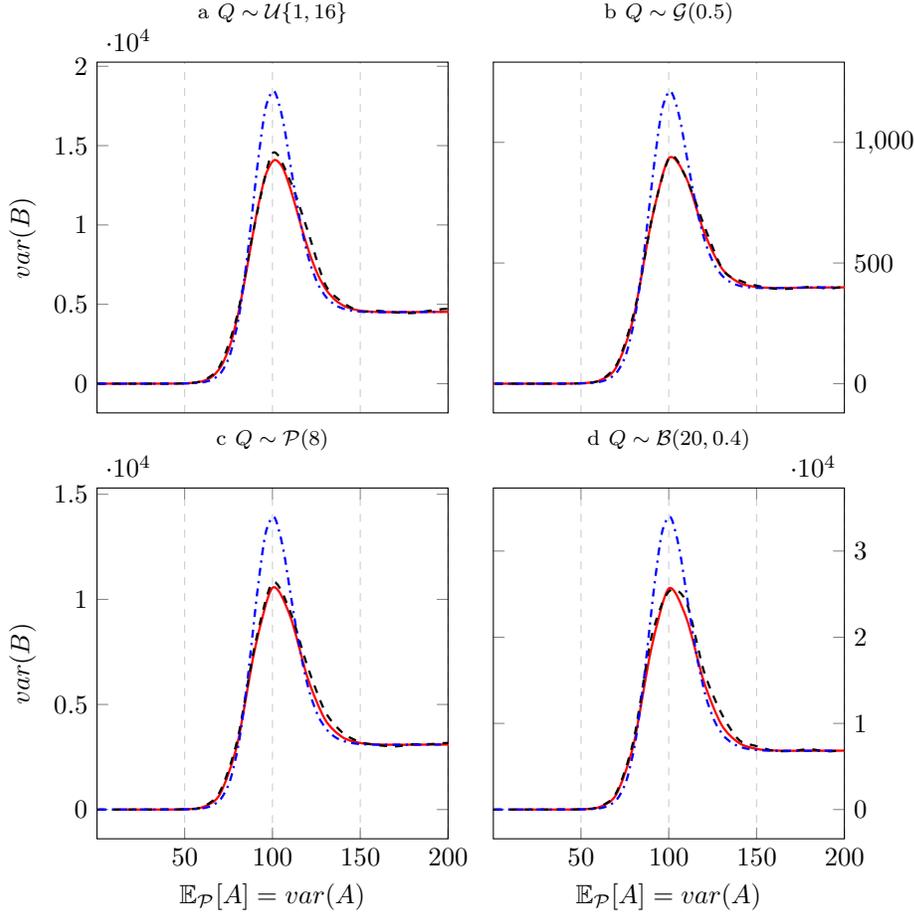
The left hand side of Fig. \ref{fig.Fig3} shows the relative error, expressed as a percentage, between the variance of $B$ and the bound corresponding to different values of {$\mathbb{E}_{\mathcal{P}}[A]$ (obtained by varying $\gamma_A$)} and $\gamma_B$ when $Q$ has {the distribution in \eqref{eq.approxGeom} with $p=0.5$.}
The right hand side of Fig.\ref{fig.Fig3} displays the relative error between the variance of $B$ and the variance obtained through \ac{LNA}. We find that the bound is still not conservative and performs consistently better than variance obtained through \ac{LNA} for different values of $\gamma_A$ and~$\gamma_B$.
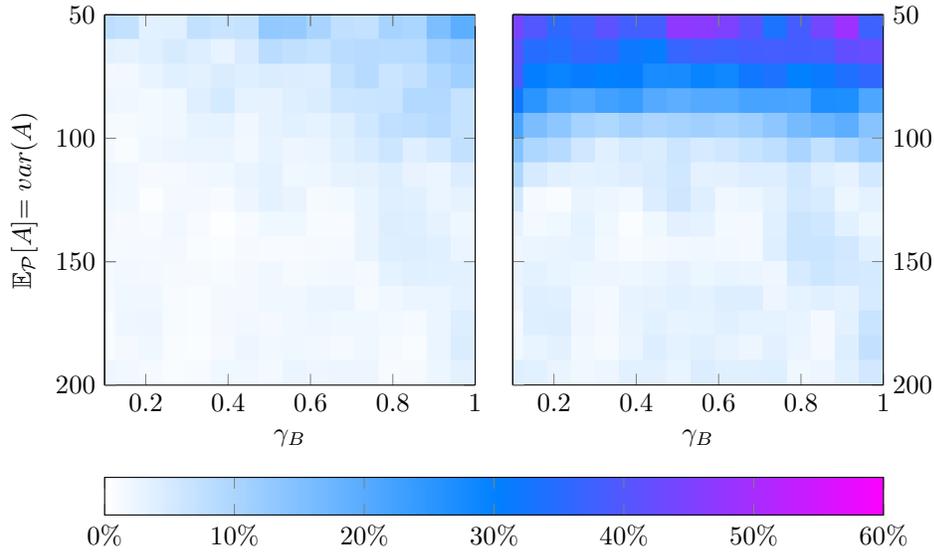
\begin{figure}[!htbp]
\centering
\begin{tikzpicture}
\begin{groupplot}[
    group style             = {group size = 2 by 1, horizontal sep = 0.5cm},
    view                    = {0}{90},
    width					= 0.5\columnwidth,
	height					= 0.5\columnwidth,
    xmin                    = 0.1,
    xmax                    = 1,
    ymin                    = 50,
    ymax                    = 200,
    y dir                   = reverse,
    point meta min          = 0,
    point meta max          = 60
    ]
\nextgroupplot[
ylabel =  {$\mathbb{E}_{\mathcal{P}}[A]{=var(A)}$}, xlabel =  {$\gamma_B$},
            colorbar,
            colormap/cool,
            colorbar horizontal,
            every colorbar/.append style={xticklabel = \pgfmathprintnumber\tick\%,width=
            2*\pgfkeysvalueof{/pgfplots/parent axis height}+
            \pgfkeysvalueof{/pgfplots/group/horizontal sep}},
]
\addplot3[surf,shader=flat]
table{Justbound.txt};
\nextgroupplot[ xlabel =  {$\gamma_B$},yticklabel pos=right,colormap/cool]
\addplot3[surf,shader=flat]
table{FdtAndVarr.txt};
\end{groupplot}
\end{tikzpicture}
\caption{{The left hand side of the figure shows the relative error between the variance of $B$, obtained through a numerical simulation, and the bound in Proposition \ref{prop.VarianceBoundNonLinearProcess}. The right hand side displays the relative error between the variance of $B$, and the \ac{LNA}. $R(a)$ is set to be a Hill function with parameters $n_h= 9$ and $A_0=100$, and $F=1$. The considered burst distribution is {\eqref{eq.approxGeom} with $p=0.5$.}
}}
\label{fig.Fig3}
\end{figure}
\section{Conclusions}\label{sec.Conclusions}
We have derived {an analytical expression that provides} a hard bound for the variance in the molecule numbers of a species formed {with bursts, with a nonlinear rate that depends on another spontaneously formed species.} 
The bound follows from {a discrete expansion based on the {\em Newton series}, and exploits spectral properties of the master equation. We have also shown that the  bound holds with equality if the propensity functions are linear. The accuracy of the bound has been investigated with numerical simulations which demonstrate} that this is very close to the actual variance also when the rate of formation of the species is {highly} nonlinear, {and despite the presence of bursts}. {We have also shown that the {\em Newton series} expansion allows to obtain an exact analytical expression for the covariance of the species under consideration. {Directions for future work include an investigation of whether the results in the paper can be extended to more broad classes of reactions, where the Newton series expansion is used as an alternative to Linear noise approximations (which are based on Taylor series expansions), for generating more accurate approximate expressions for the variance.}}

\end{document}